\documentclass[reqno,a4paper]{amsart}
\usepackage[utf8]{inputenc}

\usepackage{amsmath,amsthm,amssymb}
\usepackage{mathrsfs}
\usepackage{yhmath}

\newtheorem{theorem}{Theorem}

\newtheorem{lemma}{Lemma}

\theoremstyle{definition}
\newtheorem{definition}{Definition}

\newcommand{\C}[0]{\mathbb{C}}
\newcommand{\R}[0]{\mathbb{R}}

\newcommand{\Z}[0]{\mathbb{Z}}
\newcommand{\dif}[1]{\textnormal{d}#1}
\newcommand{\ie}[0]{\textnormal{i}}
\newcommand{\de}[0]{\mathrel{\mathop:}=}

\title[On Littlewood's proof of the prime number theorem]{On Littlewood's proof of the prime number theorem}

\author{Aleksander Simoni\v{c}}
\address{School of Science, The University of New South Wales (Canberra), ACT, Australia}
\email{a.simonic@student.adfa.edu.au}

\date{}

\begin{document}

\begin{abstract}
In this note we examine Littlewood's proof of the prime number theorem. We show that this can be extended to provide an equivalence between the prime number theorem and the non-vanishing of Riemann's zeta-function on the one-line. Our approach goes through the theory of almost periodic functions and is self-contained.
\end{abstract}

\maketitle
\thispagestyle{empty}

\section{Introduction}

The prime number theorem (PNT) is considered one of the most important theorems in mathematics. It states that $\pi(x)\sim x/\log{x}$, where $\pi(x)$ counts prime numbers less than or equal to $x$, and is equivalent to $\psi(x)\sim x$, where $\psi(x)=\sum_{p^r\leq x}\log{p}$ is the Chebyshev function, see \cite[Section I.4]{InghamDistr}. 

Apart from the Selberg and Erd\H{o}s elementary approach to the PNT, the essential part in all known proofs consists of knowing the zero free region of the Riemann zeta-function $\zeta(s)$. Denote by $\rho=\beta+\ie\gamma$ the nontrivial zeros of $\zeta(s)$, where $0\leq\beta\leq1$ and $\gamma\in\R$. A not very well-known proof of the PNT was given by Littlewood in \cite{LittlewoodPNT} where he demonstrated that it is equivalent to 
\begin{equation}
\label{eq:main2}
\lim_{x\downarrow0} \sum_{\rho} x^{1-\rho}\Gamma(\rho) = 0.
\end{equation}
Here $\Gamma(z)$ is the gamma function. The series in \eqref{eq:main2} is uniformly convergent\footnote{This simply follows from the fact that $\Gamma(x+\ie y) \ll e^{-|y|}$, valid uniformly for $x\leq 2$ and $|y|\geq5$, and $N(T+1)-N(T)=O\left(\log{T}\right)$, where $N(T)$ is a number of those zeros $\rho$ with $0<\gamma\leq T$; see \cite[Theorem 9.2]{Titchmarsh}.} and this allows us to apply the limit inside the sum. Observe that $\beta\neq1$ implies \eqref{eq:main2} and consequently the PNT. The converse statement is well-known:

\begin{theorem}
\label{thm:pntconv}
The prime number theorem implies $\beta\neq1$.
\end{theorem}

The common proof of Theorem~\ref{thm:pntconv} goes through the formula
\begin{equation}
\label{eq:converse}
-\frac{\zeta'}{\zeta}(s) - \frac{s}{s-1} = \int_{1}^{\infty} \frac{\psi(x)-x}{x^{s+1}}\dif{x}
\end{equation}
which is valid for $\Re\{s\}>1$; see \cite[p.~37]{InghamDistr}. While this idea is independent of any approach to the PNT, it is also very tempting to use \eqref{eq:main2} in the opposite direction and thus ``complete'' Littlewood's proof. 

The main purpose of this note is twofold: to sketch Littlewood's proof in hope to make it more popular, and to provide a proof of Theorem \ref{thm:pntconv} using identity \eqref{eq:main2}. Our approach goes through the theory of almost periodic functions (Definition \ref{def:apf}). We should mention that this idea is not new. In \cite[pp.~261--262]{KacPer} it was used to establish equivalence of the PNT and $\beta\neq 1$ for functions in the Selberg class. This proof is considerably more difficult than ours and use properties of almost periodic functions, e.~g., uniqueness theorem, which are not so trivial as it might appear. We show that it is possible to provide all necessary details in a concise way while avoiding the concept of the Fourier series of an almost periodic function, thus making this exposition accessible also to non-specialists. 

The outline of our proof is the following. Assume the existence of zeros $\rho=1+\ie\gamma$ and denote the ordinates of such zeros in the upper half-plane by $0<\gamma_1\leq\gamma_2\leq\cdots$. By symmetry, $\gamma_{-j}\de-\gamma_{j}$, $j>0$ are the other zeros. Let $\mathscr{S}\subseteq\Z\setminus\{0\}$ be the set of all indices of $\gamma$ (could be finite or infinite) and let $\left\{a_n\right\}_{n\in\mathscr{S}}$ be a sequence of complex numbers such that the corresponding series converges absolutely, and if $\gamma_i=\gamma_j$ then $a_i=a_j$. By the identity principle, the set $\left\{\gamma_n\right\}_{n\in\mathscr{S}}$ does not have accumulation points. Define the following function:
\begin{equation}
\label{eq:function}
F\left(x;a_n\right) \de \sum_{n\in\mathscr{S}} a_n e^{\ie\gamma_{n}x}.
\end{equation}
Then the PNT is equivalent to
\begin{equation}
\label{eq:limits}
\lim_{x\to\infty}F\left(x;\Gamma\left(1+\ie\gamma_n\right)\right)=0.
\end{equation}
In Section~\ref{sec:apf} we will show that $F(x;a_n)$ is an almost periodic function, and in Section~\ref{sec:proof} that $\lim_{x\to\infty}F(x;a_n)=0$ implies $F\equiv 0$, see Lemma \ref{lem:main}, and furthermore that this implies $a_n=0$ for every $n\in\mathscr{S}$. In view of \eqref{eq:limits} this would be a contradiction and the proof of Theorem~\ref{thm:pntconv} will be complete.

\section{Littlewood's proof of the PNT.}


Most proofs of the PNT consist of two main parts called ``Tauberian'' and ``analytical''. A Tauberian theorem deals with the question if it is possible to obtain a (partial) converse of an Abelian theorem. As an example, look at the following statement that for $a_n\geq0$
\[
   \lim_{x\downarrow0} x\sum_{n=1}^\infty a_ne^{-nx} = 1 \quad \Longleftrightarrow \quad \lim_{m\to\infty} \frac{1}{m}\sum_{n=1}^{m} a_n = 1
\]
holds. The left implication has the similar nature as the classical theorem due to Abel on a continuity of a power series at the point on a boundary of its convergence disc. The main strength lies with the right implication and this can be obtained from the celebrated Hardy--Littlewood theorem from 1914. Karamata found in 1930 a much simpler two-page proof which uses the Weierstrass approximation theorem as the only advanced tool, see \cite[pp.~226--229]{TitchFun}. If we take $a_n=\Lambda(n)$ where $\Lambda$ is the von Mangoldt function, 
then the PNT is equivalent to
\begin{equation}
\label{eq:tauber}
\lim_{x\downarrow0} x\sum_{n=1}^\infty \Lambda(n)e^{-nx} = 1.
\end{equation}
We would like to mention here that one year earlier Ramanujan studied in his third letter to Hardy the following function:
\[
   \phi(x) \de \phi_1(x)-\phi_2(x) \de \sum_{n=1}^\infty \Lambda(n)e^{-nx} - \log{2}\sum_{n=1}^\infty 2^n e^{-2^nx}.
\]
He claimed without proof that $\lim_{x\downarrow0}x\phi_2(x)=1$ and $\lim_{x\downarrow0} x\phi(x)=0$, which would consequently imply \eqref{eq:tauber}. But Hardy used in \cite[p.~39]{Har} a clever argument to show that the first limit is not only wrong but it cannot even exist, thus implying that the second limit is also wrong. On the same page Hardy wrote: \emph{``I should like to say that ``rigour apart, he found the Hardy--Littlewood proof'', but I cannot''}. The interested reader may find in this treatise some other examples of incorrect claims in analytic number theory by Ramanujan.

The $\Lambda$-function for $\Re\{s\}>1$ satisfies the important relation $-\zeta'(s)/\zeta(s)=\sum_{n=1}^\infty \Lambda(n)n^{-s}$. The analytical part of the proof begins with this equation together with the Mellin integral (see \cite[Section 2.15]{Titchmarsh}) for $e^{-x}$ to get 
\[
   x\sum_{n=1}^\infty \Lambda(n)e^{-nx} = -\frac{1}{2\pi\ie}\lim_{t\to\infty} \int_{2-\ie t}^{2+\ie t} \Gamma(s)\frac{\zeta'}{\zeta}(s)x^{1-s}\dif{s},
\]
where $x>0$. Now the idea is to take a contour integral along the rectangle $\mathcal{R}_{T}$ with vertices $-1/2\pm\ie T$ and $2\pm\ie T$ where the horizontal segments do not pass through the zeros of the zeta function\footnote{Littlewood takes $-1$ instead of $-1/2$ in the contour, but this is not a good choice because the gamma function has a pole at $-1$. We choose $-1/2$ but any number in the interval $(-1,0)$ would suffice.}. By the calculus of residues we then have 
\[
   -\frac{1}{2\pi\ie}\int_{\mathcal{R}_T} \Gamma(s)\frac{\zeta'}{\zeta}(s)x^{1-s}\dif{s} = -\sum_{\left|\Im\{\rho\}\right|<T} x^{1-\rho}\Gamma(\rho)+1-x\frac{\zeta'}{\zeta}(0).
\]
The first part clearly comes from the zeros of $\zeta(s)$ within the contour, while the second and third parts come from simple poles of $\zeta'(s)/\zeta(s)$ and $\Gamma(s)$ at $1$ and $0$, respectively. We need a result which asserts that there is an increasing and unbounded sequence $\left\{T_j\right\}_{j=1}^\infty$ of positive numbers such that $\zeta'\left(\sigma\pm\ie T_j\right)/\zeta\left(\sigma\pm\ie T_j\right)= O\left(T_j\right)$, uniformly for $-1\leq\sigma\leq2$, and this could be deduced from an approximate formula for $\zeta'(s)/\zeta(s)$; see \cite[Theorem 9.6 (A)]{Titchmarsh}. We also need an estimate $\zeta'\left(-1/2+\ie t\right)/\zeta\left(-1/2+\ie t\right)\ll |t|+1$, which follows from the logarithmic derivative version of the functional equation for $\zeta(s)$. With these estimates in hand, we could show
\begin{flalign*}
   -\frac{1}{2\pi\ie}\int_{\mathcal{R}_{T_j}} \Gamma(s)&\frac{\zeta'}{\zeta}(s)x^{1-s}\dif{s} = -\frac{1}{2\pi\ie} \int_{2-\ie T_j}^{2+\ie T_j} \Gamma(s)\frac{\zeta'}{\zeta}(s)x^{1-s}\dif{s} \\
   &+ O\left(x\sqrt{x}+e^{-T_j}\left(\frac{x^2\sqrt{x}\left(1-\log{x}\right)-1}{2x\log{x}}T_j-x\sqrt{x}\right)\right).
\end{flalign*}
Taking $j\to\infty$ in the above formula, we obtain
\[
    x\sum_{n=1}^\infty \Lambda(n)e^{-nx} = -\sum_{\rho} x^{1-\rho}\Gamma(\rho) + 1-x\frac{\zeta'}{\zeta}(0) + O\left(x\sqrt{x}\right),
\]
which finally gives 
\begin{equation}
\label{eq:analytic}
\lim_{x\downarrow0} \left(x\sum_{n=1}^\infty \Lambda(n)e^{-nx} + \sum_{\rho} x^{1-\rho}\Gamma(\rho)\right) = 1.
\end{equation}
Equation \eqref{eq:analytic} was already announced in Hardy and Littlewood's influential paper \cite{HarLit}, but used for different purposes. It is clear now that combination of \eqref{eq:tauber} and \eqref{eq:analytic} produces~\eqref{eq:main2}.

\section{Almost periodic functions}
\label{sec:apf}

The theory of almost periodic functions was initiated by H.~Bohr in 1925 and turned out to be very useful in the study of differential equations and Fourier analysis, see \cite{Besicovitch,Cor}. The space of such functions has remarkable properties. It includes the space of periodic functions, is a vector space, and the limit of every uniformly convergent sequence of almost periodic functions is also almost periodic, see Theorem \ref{thm:prop} below. The following definition is due to Bochner.

\begin{definition}
\label{def:apf}
Let $f(x)$ be a continuous function on $\R$ with complex values. We say that $f(x)$ is \emph{(uniformly) almost periodic} if for every sequence $\{x_n\}_{n=1}^\infty\subset\R$ there exists a subsequence $\left\{x_{1,n}\right\}_{n=1}^\infty$ such that $f\left(x+x_{1,n}\right)$ converges uniformly.
\end{definition}

We denote the space of all such functions by $\mathfrak{P}$. While the next theorem is fundamental, the common proof is somewhat longer than our proof, see \cite[pp.~1--5]{Besicovitch}. The reason is that we use Bochner's definition instead of Bohr's original one.

\begin{theorem}
\label{thm:prop}
The following three properties hold:
\begin{enumerate}
\item Continuous periodic functions are almost periodic.
\item If $f_1,f_2\in\mathfrak{P}$ and $a_1,a_2\in\C$, then $a_1f_1+a_2f_2\in\mathfrak{P}$.
\item If $\left\{f_n\right\}_{n=1}^\infty\subset\mathfrak{P}$ and $f_n\to f$ uniformly, then $f\in\mathfrak{P}$.
\end{enumerate}
\end{theorem}

\begin{proof}
We will provide a proof of the first property only. The second property is trivial, and a proof of the third property is straightforward if one exploits Cantor's diagonal process. 

Let $f$ be a continuous periodic function with period $\omega>0$. Take an arbitrary sequence $\{x_n\}_{n=1}^\infty\subset\R$ and define a sequence $\left\{x_n'\right\}_{n=1}^\infty$
by 
\[
x_n'\de\min\left\{x_n-k\omega\geq0\colon k\in\Z\right\}. 
\]
Then $f\left(x+x_n\right)=f\left(x+x_n'\right)$ and $\left\{x_n'\right\}_{n=1}^\infty\subset[0,\omega)$. It follows that there exists a subsequence $\left\{x_{1,n}\right\}_{n=1}^\infty$ of $\left\{x_n\right\}_{n=1}^\infty$ such that $\left\{x_{1,n}'\right\}_{n=1}^\infty$ converges to some $x_0\in[0,\omega]$. Because $f$ is continuous on a compact set $[0,2\omega]$, it is also uniformly continuous there. This means that for every $\varepsilon>0$ there exists $\delta>0$ such that $\left|x_{1,n}'-x_0\right|<\delta$ implies $\left|f\left(x+x_{1,n}'\right)-f\left(x+x_0\right)\right|<\varepsilon$ for every $x\in[0,\omega]$. But then for every $\varepsilon>0$ there exists $N$ such that $\left|f\left(x+x_{1,n}\right)-f\left(x+x_0\right)\right|<\varepsilon$ for every $n>N$ and $x\in\R$. Therefore, $f\left(x+x_{1,n}\right)\to f\left(x+x_0\right)$ uniformly and $f$ is thus almost periodic.
\end{proof}

\section{Proof of Theorem \ref{thm:pntconv}}
\label{sec:proof}

Like a periodic function, an almost periodic function has the property that the existence of its limit at infinity characterizes it completely.

\begin{lemma}
\label{lem:main}
Let $f(x)$ be an almost periodic function. Then $f(x)\equiv C\in\C$ if and only if $\lim_{x\to\infty}f(x)=C$.
\end{lemma}

\begin{proof}
If $f(x)\equiv C$, then $\lim_{x\to\infty}f(x)=C$. On the contrary, assume that $f(x)$ is not a constant function. Then there exist $x_0$ and $y_0$ such that $f\left(x_0\right)\neq f\left(y_0\right)$. Define $m\de\frac{1}{3}\left|f\left(x_0\right)-f\left(y_0\right)\right|>0$. Because $f\in\mathfrak{P}$, there exists a strictly increasing sequence $\{h_n\}_{n=1}^\infty$ with $h_n\to\infty$ such that $f\left(x+h_n\right)$ converges uniformly. This means that there exists $N$ such that $\left|f\left(x+h_N\right)-f\left(x+h_n\right)\right|<m$ for all $n>N$ and $x\in\R$. For $n>N$ define $x_n\de x_0+h_n-h_N$ and $y_n\de y_0+h_n-h_N$. Then $x_n\to\infty$ and $y_n\to\infty$. We have
\begin{flalign*}
3m &= \left|f\left(x_0\right)-f\left(y_0\right)\right| \\
   &\leq \left|f\left(x_0\right)-f\left(x_n\right)\right| + \left|f\left(x_n\right)-f\left(y_n\right)\right| + \left|f\left(y_n\right)-f\left(y_0\right)\right| \\
   &< 2m + \left|f\left(x_n\right)-f\left(y_n\right)\right|;
\end{flalign*}
therefore $\left|f\left(x_n\right)-f\left(y_n\right)\right|>m$. This means that $\lim_{x\to\infty}f(x)$ could not exist. So if $\lim_{x\to\infty}f(x)=C$, then $f$ must be a constant function and it is clear that this constant is $C$.
\end{proof}

\begin{proof}[Proof of Theorem \ref{thm:pntconv}]
All three properties in Theorem \ref{thm:prop} guarantee that $F\left(x;a_n\right)$ is an almost periodic function. If $\lim_{x\to\infty}F\left(x;a_n\right)=0$, then Lemma \ref{lem:main} implies that $F\left(x;a_n\right)=0$ for all $x\in\R$. Take some $n_0\in\mathscr{S}$ and let $\mathscr{S}_{0}$ be a subset of $\mathscr{S}$ such that $n\in\mathscr{S}_{0}$ if and only if $\gamma_n=\gamma_{n_0}$. We have
\[
   F\left(x;a_n\right)e^{-\ie\gamma_{n_0}x} = \left|\mathscr{S}_{0}\right|a_{n_0} + \sum_{n\in\mathscr{S}\setminus\mathscr{S}_{0}} a_ne^{\ie\left(\gamma_{n}-\gamma_{n_0}\right)x}.
\]
Because we can change the order of summation and integration, the above equation implies
\[
   I(X)\de\frac{1}{X}\int_{0}^{X} F\left(x;a_n\right)e^{-\ie\gamma_{n_0}x}\dif{x} = \left|\mathscr{S}_{0}\right|a_{n_0} + \frac{1}{\ie X}R(X) 
\]
where
\[
   R(X) \de \sum_{n\in\mathscr{S}\setminus\mathscr{S}_{0}} \frac{a_n}{\gamma_n-\gamma_{n_0}}\left(e^{\ie\left(\gamma_{n}-\gamma_{n_0}\right)X}-1\right).
\]
We also observe that $ \lim_{X\to\infty} I(X)=\left|\mathscr{S}_{0}\right|a_{n_0}$ since
\[
   \left|R(X)\right|\leq\frac{2}{d}\sum_{n\in\mathscr{S}}\left|a_n\right|<\infty
\]
where $d\de\min_{n\in\mathscr{S}\setminus\mathscr{S}_{0}}\left\{\left|\gamma_n-\gamma_{n_0}\right|\right\}>0$. But then $F\left(x;a_n\right)\equiv0$ implies $a_n=0$ for all $n\in\mathscr{S}$. Consequently, the limit \eqref{eq:limits} does not hold and the Riemann zeta function does not have zeros with real parts equal to one.
\end{proof}

Finally, we point out that it is possible to construct the theory of almost periodic functions through trigonometric polynomials $T(x)=\sum_{k=1}^n c_ke^{\ie\lambda_{k}x}$, where $c_k$ are complex numbers and $\lambda_k$ are real numbers, see \cite{Cor}. Then we could say that $f(x)$ is an almost periodic function when for every $\varepsilon>0$ there exists a trigonometric polynomial $T_{\varepsilon}(x)$ such that $\left|f(x)-T_{\varepsilon}(x)\right|<\varepsilon$ for every $x\in\R$. By this definition our function $F\left(x;a_n\right)$ is of course almost periodic, but the author could not find a similar argument as in the proof of Lemma~\ref{lem:main} by using this definition. It is equivalent to Bochner's, but the proof is somehow longer than the proof of Theorem~\ref{thm:prop}.

\subsection*{Acknowledgements}
The author is grateful to Tim Trudgian for helpful and constructive remarks on the manuscript.

\ifx\undefined\bysame
\newcommand{\bysame}{\leavevmode\hbox to3em{\hrulefill}\,}
\fi


\begin{thebibliography}{Har40}
\bibitem[Bes55]{Besicovitch}
A.~S. Besicovitch, {\em Almost periodic functions}, Dover Publications, Inc.,
  New York, 1955.

\bibitem[Cor89]{Cor}
C.~Corduneanu, {\em Almost periodic functions}, Chelsea Publishing Company, New
  York, 1989.

\bibitem[Har40]{Har}
G.~H. Hardy, {\em Ramanujan. {T}welve lectures on subjects suggested by his
  life and work}, Cambridge University Press, Cambridge, England; Macmillan
  Company, New York, 1940.

\bibitem[HL16]{HarLit}
G.~H. Hardy and J.~E. Littlewood, {\em Contributions to the theory of the
  {R}iemann zeta-function and the theory of the distribution of primes}, Acta
  Math. {\bf 41} (1916), no.~1, 119--196.

\bibitem[Ing90]{InghamDistr}
A.~E. Ingham, {\em The distribution of prime numbers}, Cambridge Mathematical Library, Cambridge University Press,
  Cambridge, 1990.

\bibitem[KP03]{KacPer}
J.~Kaczorowski and A.~Perelli, {\em On the prime number theorem for the
  {S}elberg class}, Arch. Math. (Basel) {\bf 80} (2003), no.~3, 255--263.

\bibitem[Lit71]{LittlewoodPNT}
J.~E. Littlewood, {\em The quickest proof of the prime number theorem}, Acta
  Arith. {\bf 18} (1971), 83--86.

\bibitem[Tit58]{TitchFun}
E.~C. Titchmarsh, {\em The theory of functions}, Oxford University Press,
  Oxford, 1958.

\bibitem[Tit86]{Titchmarsh}
E.~C. Titchmarsh, {\em The theory of the {R}iemann zeta-function}, 2nd ed., The
  Clarendon Press, Oxford University Press, New York, 1986.
\end{thebibliography}
\end{document}